\documentclass[11pt]{article}

\usepackage[T1]{fontenc}
\usepackage{lmodern}
\usepackage{microtype}
\usepackage[a4paper,margin=1in]{geometry}
\usepackage{amsmath,amssymb,amsthm,mathtools}
\usepackage{enumitem}
\usepackage{hyperref}
\usepackage[nameinlink,capitalise]{cleveref}
\hypersetup{
  colorlinks=true,
  linkcolor=blue,
  citecolor=blue,
  urlcolor=blue,
  pdftitle={Dynamical Dimension and Shift Embeddability without the Marker Property},
  pdfauthor={Ruxi Shi}
}

\newtheorem{theorem}{Theorem}[section]

\newtheorem{lemma}[theorem]{Lemma}
\newtheorem{proposition}[theorem]{Proposition}

\theoremstyle{definition}
\newtheorem{definition}[theorem]{Definition}

\theoremstyle{remark}
\newtheorem{remark}[theorem]{Remark}

\newcommand{\Sone}{\mathbb{S}}
\newcommand{\mdim}{\operatorname{mdim}}
\newcommand{\Ddim}{\operatorname{Ddim}}
\newcommand{\ord}{\operatorname{ord}}
\newcommand{\Cov}{\operatorname{Cov}}
\newcommand{\Prob}{\operatorname{Prob}}
\newcommand{\id}{\operatorname{id}}
\newcommand{\calX}{\mathcal X}

\title{Dynamical Dimension and Shift Embeddability without the Marker Property}
\author{Ruxi Shi\thanks{This work was supported by the National Natural
Science Foundation of China under Grant Nos.~12571198 and 12231013.}}
\date{}

\begin{document}
\maketitle

\begin{abstract}
We study the aperiodic inverse-limit system constructed in our earlier work as
an example of a finite-mean-dimensional dynamical system without the marker
property.  We prove that its mean dimension and Meyerovitch's dynamical
dimension are both equal to $N$.  Despite the absence of the marker property,
the system admits an equivariant topological embedding into the cubical shift
with alphabet dimension $3N+2$.  As an auxiliary result, we prove that the
dynamical dimension of the full shift over any compact metrizable alphabet is
exactly the covering dimension of the alphabet, including when this dimension
is infinite.
\end{abstract}



\section{Introduction}\label{sec:introduction}

Throughout, a \emph{topological dynamical system} $(X,T)$ consists of a
nonempty compact metrizable space $X$ and a homeomorphism $T:X\to X$.  A
\emph{factor map} $p:(Y,S)\to(Z,R)$ is a continuous surjection satisfying
$p\circ S=R\circ p$.  A continuous map
$f:X\to[0,1]^D$ determines the equivariant orbit map
\[
I_f:X\longrightarrow([0,1]^D)^{\mathbb Z},
\qquad I_f(x)=(f(T^nx))_{n\in\mathbb Z}.
\]
The shift-embedding problem asks when $f$ can be chosen so that $I_f$ is
injective.  This problem links topological dynamics with classical dimension
theory; see, for example, \cite{Coornaert2005}.  Mean dimension, introduced by
Gromov \cite{Gromov1999} and developed systematically by Lindenstrauss and
Weiss \cite{LW2000}, provides a fundamental obstruction: an embedding into the
$D$-cubical shift forces $\mdim(X,T)\le D$.

A substantial positive theory has been developed under hypotheses that
provide orbit markers.  Lindenstrauss proved an embedding theorem for systems
with an infinite minimal factor \cite{Lindenstrauss1999}.  Gutman and Tsukamoto
obtained the sharp bound for extensions of aperiodic subshifts
\cite{GutmanTsukamoto2014}, and Gutman, Qiao and Tsukamoto established the
corresponding sharp theorem for systems with the marker property
\cite{GutmanQiaoTsukamoto2019}.  The threshold $D/2$ in these results is
optimal by the examples of Lindenstrauss and Tsukamoto
\cite{LindenstraussTsukamoto2014}.  Related embedding results involving
periodic points and local markers were proved by Gutman \cite{Gutman2017}.
These results demonstrate the strength of the marker property, but
aperiodicity alone does not imply it.

Nor do finite mean dimension and freeness alone guarantee shift embeddability.
Dranishnikov and Levin constructed a free $\mathbb Z$-action by isometries on a
compact metric space that does not embed into any finite-dimensional cubical
shift \cite{DranishnikovLevin2025}.  Motivated in part by this phenomenon,
Meyerovitch introduced a new dimension invariant for dynamical systems
\cite{Meyerovitch2026}.  This invariant is monotone under equivariant
embeddings, dominates mean dimension, and detects the Dranishnikov--Levin
obstruction; moreover, its finiteness implies almost shift embeddability.  We
denote it by $\Ddim$ to distinguish it from covering dimension and mean
dimension.

In our earlier work \cite{Shi2021}, we constructed an aperiodic system of
finite mean dimension without the marker property.  It is therefore outside
the direct scope of the embedding theorems above, but this alone does not
decide whether it is shift-embeddable.  The purpose of the present paper is to
settle this question affirmatively for that system and, at the same time, to
compute both its mean dimension and its dynamical dimension exactly.  The
full-shift computation for dynamical dimension, which is needed for the upper
bound, may also be of independent interest.

Fix the parameters $N$ and $\delta$ used in the construction of \cite{Shi2021},
with $0<\delta<1$, and let $(\calX,T)$ denote the resulting inverse-limit
system, recalled in \cref{sec:construction}.  Our main result is the following.

\begin{theorem}\label{thm:main}
The dynamical system $(\calX,T)$ has the following properties:
\begin{enumerate}[label=(\roman*)]
\item $(\calX,T)$ is aperiodic;
\item $(\calX,T)$ does not have the marker property;
\item
\[
\mdim(\calX,T)=\Ddim(\calX,T)=N;
\]
\item there is an equivariant topological embedding
\[
(\calX,T)
\hookrightarrow
\left(([0,1]^{3N+2})^{\mathbb Z},\sigma\right).
\]
\end{enumerate}
\end{theorem}

\begin{remark}
The integer $N$ is selected existentially in equation~(5.1) of
\cite{Shi2021}; no canonical or minimal value is specified there.  Thus the
equality $\mdim(\calX,T)=\Ddim(\calX,T)=N$ is exact for each system produced
by the construction, rather than a universal numerical value independent of
the choices.
\end{remark}

We briefly describe the proof.  First, we show that every finite stage has
mean dimension exactly $N$.  The upper bound comes from its inclusion in the
full shift over $\Sone^N$, while a finite-block subsystem gives the lower
bound.  For the inverse limit, the upper bound follows from the standard
inverse-limit inequality, and continuous sections of the natural projections
give the lower bound.  Consequently,
\[
\mdim(\calX,T)=N.
\]

For dynamical dimension, we prove the full-shift formula
\[
\Ddim(K^{\mathbb Z},\sigma)=\dim K
\]
for every compact metrizable space $K$, together with an inverse-limit
inequality for $\Ddim$.  Applied to the stages of $\calX$, these results and
the general inequality $\mdim\le\Ddim$ yield
\[
\Ddim(\calX,T)=N.
\]
Meyerovitch's almost-embedding theorem therefore provides an equivariant
almost embedding of $(\calX,T)$ into the $(2N+1)$-cubical shift.

Finally, we prove that the first-coordinate projection of the inverse limit
$\pi_1:\calX\to X_1$ is a distal extension.  A topological embedding
$\Sone^N\hookrightarrow[0,1]^{N+1}$ induces an equivariant embedding of
$X_1$ into the $(N+1)$-cubical shift.  Recording this factor together with the
almost embedding makes the resulting map injective: distality of the extension
rules out every nontrivial collision remaining in a fiber of $\pi_1$.  This
relative-distal argument does not require the whole system $(\calX,T)$ to be
distal, and gives the claimed embedding into the $(3N+2)$-cubical shift.

The paper is organized as follows.  \Cref{sec:construction} recalls the
inverse-limit construction.  \Cref{sec:mdim-prelim,sec:finite-mdim,sec:inverse-mdim}
compute the mean dimension of the finite stages and of the inverse limit.
\Cref{sec:Ddim-inverse} recalls dynamical dimension and proves an inverse-limit
estimate.  \Cref{sec:Ddim-full-shift} proves the full-shift formula and
computes the invariant of our system.  \Cref{sec:relative-distal,sec:upgrade}
show that the first-coordinate projection is a distal extension and use it to
upgrade the almost embedding to a topological embedding.
\Cref{sec:open-questions} records several open questions suggested by the
construction.

\section{The construction}\label{sec:construction}

Let
\[
\Sone=\mathbb R/2\mathbb Z
\]
with the metric
\[
\rho(s,t)=\min_{n\in\mathbb Z}|s-t-2n|.
\]
Thus $\Sone$ is a circle of circumference $2$ and diameter $1$.  On $\Sone^N$ use the maximum metric
\[
\rho_N(x,y)=\max_{1\le j\le N}\rho(x_j,y_j).
\]
For positive integers $N,r$ and $\delta>0$, consider the subshift
\[
X(N,r,\delta)
=
\left\{x=(x_k)_{k\in\mathbb Z}\in(\Sone^N)^{\mathbb Z}:
\rho_N(x_k,x_{k+r})\ge \delta\text{ for every }k\in\mathbb Z
\right\},
\]
with the left shift $\sigma(x)_k=x_{k+1}$.

In Section~5 of our earlier paper \cite{Shi2021}, we constructed this system to
obtain a finite-mean-dimensional system without the marker property.
The failure of the marker property was deduced there from a periodic-coindex
obstruction.  In that construction, we chose an integer $N$ and a number
$\delta>0$ so that the required periodic-coindex inequality holds.  We may
additionally assume that
\[
0<\delta<1.
\]
Indeed, replacing $\delta$ by a smaller positive number enlarges
$X(N,1,\delta)$, so the coindex inequality is preserved by monotonicity under
equivariant maps.

Set
\[
q(m)=\prod_{j=1}^{m}j=m!,
\qquad
X_m=X(N,q(m),\delta),
\qquad
T_m=\sigma.
\]
For $m\ge2$, the bonding map $\theta_{m,m-1}:X_m\to X_{m-1}$ is
\begin{equation}\label{eq:bonding}
(\theta_{m,m-1}x)_k
=
\sum_{i=0}^{m-1}x_{k+i q(m-1)},
\end{equation}
where the sum is taken in the compact abelian group $\Sone^N$.  Write
\[
\theta_{m,n}=\theta_{n+1,n}\circ\cdots\circ\theta_{m,m-1}
\qquad(m>n).
\]
We also set $\theta_{m,m}=\id_{X_m}$.
To see directly that \eqref{eq:bonding} maps $X_m$ into $X_{m-1}$, note that
translation invariance of $\rho_N$ and the identity $q(m)=m q(m-1)$ give
\[
\rho_N\bigl((\theta_{m,m-1}x)_k,
(\theta_{m,m-1}x)_{k+q(m-1)}\bigr)
=
\rho_N(x_k,x_{k+q(m)})
\ge\delta.
\]
The formula in \eqref{eq:bonding} also defines a continuous group
homomorphism between the ambient full shifts.  We use the same notation
$\theta_{m,m-1}$, and hence $\theta_{m,n}$, for these extensions; on $X_m$
they agree with the bonding maps above.
In \cite[Lemmas~5.2--5.3]{Shi2021}, we constructed continuous right inverses
\[
\eta_{m-1,m}:X_{m-1}\longrightarrow X_m,
\qquad
\theta_{m,m-1}\circ\eta_{m-1,m}=\id_{X_{m-1}},
\]
and hence all bonding maps are surjective.  The inverse limit is
\[
(\calX,T)=\varprojlim (X_m,T_m),
\]
with natural projections $\pi_m:\calX\to X_m$.  As proved in \cite{Shi2021},
this system does not have the marker property.  It is also aperiodic.  Indeed,
if $T^p\mathbf x=\mathbf x$ for some $p\ge1$, then for any $m$ with
$p\mid q(m)$ the coordinate $x^{(m)}$ is $q(m)$-periodic, contradicting the
defining inequality
$\rho_N(x^{(m)}_k,x^{(m)}_{k+q(m)})\ge\delta>0$.

\section{Preliminaries}\label{sec:mdim-prelim}

We first record the cover-theoretic definition of mean dimension and the
standard properties needed to analyze the finite stages of the inverse
system.

For a finite open cover $\alpha$ of a compact space $Y$, let
\[
\ord(\alpha)=\sup_{y\in Y}\left(-1+\sum_{U\in\alpha}\mathbf 1_U(y)\right),
\]
and let
\[
D_Y(\alpha)=\inf_{\beta\succeq\alpha}\ord(\beta),
\]
where the infimum is over finite open covers $\beta$ of $Y$, and
$\beta\succeq\alpha$ means that $\beta$ refines $\alpha$.  For a
homeomorphism $S:Y\to Y$,
\[
\mdim(Y,S)
=
\sup_{\alpha}
\lim_{n\to\infty}
\frac{1}{n}
D_Y\left(\bigvee_{j=0}^{n-1}S^{-j}\alpha\right).
\]
The limit exists by the standard subadditivity argument.  We use the following
standard facts from mean-dimension theory; see, for example,
\cite{LW2000,Coornaert2005}.

\begin{enumerate}[label=(\roman*)]
\item If $Z\subseteq Y$ is closed and $S$-invariant, then
\[
\mdim(Z,S|_Z)\le \mdim(Y,S).
\]
\item For every positive integer $k$,
\begin{equation}\label{eq:iterate-scaling}
\mdim(Y,S^k)=k\,\mdim(Y,S).
\end{equation}
\item For $D\in\mathbb N$,
\begin{equation}\label{eq:cube-shift-mdim}
\mdim\left(([0,1]^D)^{\mathbb Z},\sigma\right)=D.
\end{equation}
More generally, if $K$ is finite-dimensional, then
\[
\mdim(K^{\mathbb Z},\sigma)\le \dim K.
\]
Tsukamoto determined the exact value \cite[Theorem~1.1]{Tsukamoto2019}: it is
$\dim K$ when $K$ is of basic type (i.e.,
$\dim(K\times K)=2\dim K$), and $\dim K-1$ when $K$ is of exceptional type.
Since $\Sone^N$ is of basic type, in particular,
\begin{equation}\label{eq:torus-upper}
\mdim\left((\Sone^N)^{\mathbb Z},\sigma\right)=N.
\end{equation}
\end{enumerate}

For completeness, the scaling identity in \eqref{eq:iterate-scaling} follows directly from the cover definition.  The inequality
$\mdim(Y,S^k)\le k\mdim(Y,S)$ follows because the join over times $0,k,2k,\ldots,(n-1)k$ is coarser than the join over all times $0,1,\ldots,kn-1$.  Conversely, for a cover $\alpha$, put
\[
\beta=\bigvee_{i=0}^{k-1}S^{-i}\alpha.
\]
Then
\[
\bigvee_{j=0}^{n-1}(S^k)^{-j}\beta
=
\bigvee_{i=0}^{kn-1}S^{-i}\alpha,
\]
which yields the reverse inequality after taking limits and suprema.

We will also use the following inverse-limit estimate
\cite[Proposition~5.8]{Shi2021}.

\begin{proposition}\label{prop:mdim-inverse}
Let
\[
(Y,S)=\varprojlim(Y_m,S_m)
\]
be an inverse limit of topological dynamical systems with surjective
bonding maps, and let $p_m:Y\to Y_m$ be the natural projections.  Then
\[
\mdim(Y,S)\le \sup_m\mdim(Y_m,S_m).
\]
\end{proposition}

\section{Every finite stage has mean dimension exactly \texorpdfstring{$N$}{N}}\label{sec:finite-mdim}

We first compute the mean dimension of the finite stages.  The upper bound follows from
the ambient torus shift, while the lower bound is obtained from an explicit
full-shift subsystem.

\begin{proposition}\label{prop:finite-stage}
For every positive integer $r$ and every $0<\delta<1$,
\[
\mdim(X(N,r,\delta),\sigma)=N.
\]
In particular,
\[
\mdim(X_m,T_m)=N
\qquad\text{for every }m\ge1.
\]
\end{proposition}

\begin{proof}
The upper bound follows from the inclusion
\[
X(N,r,\delta)\subseteq(\Sone^N)^{\mathbb Z}
\]
and \eqref{eq:torus-upper}:
\[
\mdim(X(N,r,\delta),\sigma)\le N.
\]

We prove the reverse inequality by placing a large cubical full shift inside the system for the iterate $\sigma^{2r}$.  Choose
\[
0<\varepsilon<\min\left\{\frac12,\frac{1-\delta}{2}\right\}.
\]
Let
\[
I=\{t\in\Sone:\rho(t,0)\le\varepsilon\},
\qquad
J=\{t\in\Sone:\rho(t,1)\le\varepsilon\}.
\]
Both $I$ and $J$ are closed arcs homeomorphic to $[0,1]$.  Define
\[
A=I^N,
\qquad
B=J\times I^{N-1}.
\]
Then $A$ and $B$ are each homeomorphic to $[0,1]^N$.  If $a\in A$ and $b\in B$, the first coordinates satisfy
\[
\rho(a_1,b_1)
\ge \rho(0,1)-\rho(a_1,0)-\rho(b_1,1)
\ge 1-2\varepsilon
>\delta.
\]
Therefore
\begin{equation}\label{eq:AB-separated}
\rho_N(a,b)>\delta
\qquad(a\in A,\ b\in B).
\end{equation}

Let $Y_r$ be the set of all $x\in(\Sone^N)^{\mathbb Z}$ such that, for every $k\in\mathbb Z$ and every $0\le j<r$,
\[
x_{2rk+j}\in A,
\qquad
x_{2rk+r+j}\in B.
\]
This is a closed $\sigma^{2r}$-invariant set.  For every integer $n$, the two positions $n$ and $n+r$ lie in opposite members of the pair $A,B$.  Hence \eqref{eq:AB-separated} gives
\[
Y_r\subseteq X(N,r,\delta).
\]

Group the coordinates into consecutive blocks of length $2r$.  The map
\[
\Psi:Y_r\longrightarrow (A^r\times B^r)^{\mathbb Z},
\qquad
\Psi(x)_k=(x_{2rk},x_{2rk+1},\ldots,x_{2rk+2r-1}),
\]
is a topological conjugacy from $(Y_r,\sigma^{2r})$ to the full shift on the alphabet $A^r\times B^r$.  Since
\[
A^r\times B^r\cong [0,1]^{2rN},
\]
we obtain from \eqref{eq:cube-shift-mdim} that
\[
\mdim(Y_r,\sigma^{2r})=2rN.
\]
By subsystem monotonicity and the iterate-scaling identity,
\[
2rN
\le
\mdim(X(N,r,\delta),\sigma^{2r})
=
2r\,\mdim(X(N,r,\delta),\sigma).
\]
Thus
\[
\mdim(X(N,r,\delta),\sigma)\ge N.
\]
Together with the upper bound, this proves equality.
\end{proof}

\section{Mean dimension of \texorpdfstring{$(\calX,T)$}{(X,T)}}\label{sec:inverse-mdim}

Having computed the mean dimension of every finite stage, we now pass to the
inverse limit.  We first prove a cover-theoretic lemma explaining why a
nonequivariant section suffices for a mean-dimension lower bound.

\begin{lemma}\label{lem:section-cover}
Let $p:Y\to Z$ be a continuous surjection between compact Hausdorff spaces.  Suppose that there is a continuous map $s:Z\to Y$ with $p\circ s=\id_Z$.  Then, for every finite open cover $\alpha$ of $Z$,
\begin{equation}\label{eq:D-equality}
D_Y(p^{-1}\alpha)=D_Z(\alpha).
\end{equation}
If, in addition, $p:(Y,S)\to(Z,R)$ is equivariant, then
\[
\mdim(Y,S)\ge \mdim(Z,R).
\]
The section $s$ need not be equivariant.
\end{lemma}

\begin{proof}
Pulling back refinements gives
\[
D_Y(p^{-1}\alpha)\le D_Z(\alpha).
\]
For the reverse inequality, let $\beta$ be any finite open refinement of
$p^{-1}\alpha$ on $Y$.  Consider the finite open cover
\[
s^{-1}\beta=\{s^{-1}(B):B\in\beta\}
\]
of $Z$, with empty members omitted.  To see that it refines $\alpha$, fix
$B\in\beta$.  Since $\beta\succeq p^{-1}\alpha$, there is some
$U_B\in\alpha$ such that $B\subseteq p^{-1}(U_B)$.  Therefore
\[
s^{-1}(B)\subseteq s^{-1}\bigl(p^{-1}(U_B)\bigr)
=(p\circ s)^{-1}(U_B)=U_B.
\]
By the definition of inverse image, for every $z\in Z$ and $B\in\beta$,
\[
z\in s^{-1}(B)\quad\Longleftrightarrow\quad s(z)\in B.
\]
Thus the map $B\mapsto s^{-1}(B)$ sends the members of $\beta$ containing
$s(z)$ onto the members of $s^{-1}\beta$ containing $z$.  (Different members
of $\beta$ may have the same inverse image.)  Hence $s^{-1}\beta$ is an open
refinement of $\alpha$ whose order is no larger than that of $\beta$.
By the definition of $D_Z(\alpha)$,
\[
D_Z(\alpha)\le \ord(s^{-1}\beta)\le \ord(\beta).
\]
Since $\beta$ was an arbitrary finite open refinement of
$p^{-1}\alpha$ on $Y$, taking the infimum over all such $\beta$ gives
\[
D_Z(\alpha)
\le
\inf_{\beta\succeq p^{-1}\alpha}\ord(\beta)
=D_Y(p^{-1}\alpha).
\]
Together with the opposite inequality proved above, this yields
\eqref{eq:D-equality}.

If $p$ is equivariant, then
\[
\bigvee_{j=0}^{n-1}S^{-j}p^{-1}\alpha
=
p^{-1}\left(\bigvee_{j=0}^{n-1}R^{-j}\alpha\right).
\]
Applying \eqref{eq:D-equality} with
$\bigvee_{j=0}^{n-1}R^{-j}\alpha$ in place of $\alpha$ gives, for every
$n\ge1$,
\[
D_Y\left(\bigvee_{j=0}^{n-1}S^{-j}p^{-1}\alpha\right)
=
D_Z\left(\bigvee_{j=0}^{n-1}R^{-j}\alpha\right).
\]
After dividing by $n$ and taking the limit, we obtain
\[
\lim_{n\to\infty}\frac{1}{n}
D_Y\left(\bigvee_{j=0}^{n-1}S^{-j}p^{-1}\alpha\right)
=
\lim_{n\to\infty}\frac{1}{n}
D_Z\left(\bigvee_{j=0}^{n-1}R^{-j}\alpha\right).
\]
The left-hand side is one of the quantities over which the supremum defining
$\mdim(Y,S)$ is taken, because $p^{-1}\alpha$ is a finite open cover of
$Y$.  Hence, taking the supremum over all finite open covers $\alpha$ of
$Z$, we get
\[
\mdim(Y,S)
\ge
\sup_{\alpha}
\lim_{n\to\infty}\frac{1}{n}
D_Z\left(\bigvee_{j=0}^{n-1}R^{-j}\alpha\right)
=\mdim(Z,R).
\]
\end{proof}

We now construct the sections needed to apply the lemma to the natural
projections.  The right
inverses of the bonding maps constructed in our earlier paper provide these
sections.  More precisely, for $n>m$ let
\[
\eta_{m,n}=\eta_{n-1,n}\circ\cdots\circ\eta_{m,m+1}.
\]
Then define
\begin{equation}\label{eq:gamma}
\gamma_m:X_m\longrightarrow\calX,
\qquad
\gamma_m(x)=
(\theta_{m,1}x,\ldots,\theta_{m,m-1}x,x,
\eta_{m,m+1}x,\eta_{m,m+2}x,\ldots).
\end{equation}
The compatibility relations and the identities
$\theta_{j,j-1}\eta_{j-1,j}=\id$ show that \eqref{eq:gamma} takes values in the inverse limit.  We proved the continuity of the maps $\eta_{j-1,j}$ and hence of $\gamma_m$ in Lemmas~5.2--5.3 and the paragraph preceding Lemma~5.7 of \cite{Shi2021}.  By construction,
\begin{equation}\label{eq:section-pim}
\pi_m\circ\gamma_m=\id_{X_m}.
\end{equation}
The map $\gamma_m$ is generally not equivariant, but \cref{lem:section-cover} shows that equivariance of the section is unnecessary.

\begin{proposition}\label{prop:inverse-mdim}
The dynamical system satisfies
\[
\mdim(\calX,T)=N.
\]
\end{proposition}

\begin{proof}
Since $(\calX,T)$ is the inverse limit of the systems $(X_m,T_m)$,
\cref{prop:mdim-inverse} gives
\[
\mdim(\calX,T)
\le
\sup_m\mdim(X_m,T_m)
=N,
\]
where the last equality is \cref{prop:finite-stage}.

For the reverse inequality, fix $m$.  The projection
\[
\pi_m:(\calX,T)\longrightarrow(X_m,T_m)
\]
is equivariant and, by \eqref{eq:section-pim}, has a continuous section.  Therefore \cref{lem:section-cover} gives
\[
\mdim(\calX,T)\ge\mdim(X_m,T_m)=N.
\]
Combining the two inequalities proves the result.
\end{proof}

\section{Meyerovitch's dynamical dimension and inverse limits}\label{sec:Ddim-inverse}

We recall Meyerovitch's dynamical dimension and the properties needed below
\cite{Meyerovitch2026}.  Its finiteness yields an almost embedding into a
finite-dimensional cubical shift, while its monotonicity provides an
obstruction to equivariant embeddings.

For a finite open cover $\mathcal U$ of a dynamical system $(Y,S)$, set
\[
\ord(\mathcal V,y)
=-1+\sum_{V\in\mathcal V}\mathbf 1_V(y)
\]
and
\[
\Ddim(\mathcal U;Y,S)
=
\inf_{\mathcal V\succeq\mathcal U}
\sup_{\mu\in\Prob(Y,S)}
\int_Y \ord(\mathcal V,y)\,d\mu(y),
\]
where $\Prob(Y,S)$ is the set of $S$-invariant Borel probability measures and
the infimum is over finite open covers $\mathcal V$ refining $\mathcal U$.
Writing $\Cov(Y)$ for the set of finite open covers of $Y$, define
\[
\Ddim(Y,S)=\sup_{\mathcal U\in\Cov(Y)}\Ddim(\mathcal U;Y,S).
\]
We use the following results from \cite{Meyerovitch2026}:
\begin{enumerate}[label=(\roman*)]
\item $\Ddim$ is monotone under equivariant topological embeddings
\cite[Proposition~3.1]{Meyerovitch2026};
\item
\begin{equation}\label{eq:Ddim-fixed-point}
\dim\operatorname{Fix}(S)\le\Ddim(Y,S),
\qquad
\operatorname{Fix}(S)=\{y\in Y:S y=y\};
\end{equation}
this fixed-point lower bound follows from the argument in the proof of
\cite[Theorem~8.1]{Meyerovitch2026};
\item
\begin{equation}\label{eq:Ddim-cube}
\Ddim\left(([0,1]^d)^{\mathbb Z},\sigma\right)=d;
\end{equation}
this is \cite[Theorem~8.1]{Meyerovitch2026};
\item for every $\mathbb Z$-system,
$\mdim(Y,S)\le\Ddim(Y,S)$
\cite[Theorem~9.1]{Meyerovitch2026};
\item if $\Ddim(Y,S)<d/2$, then there
is an equivariant almost embedding
\[
(Y,S)\longrightarrow(([0,1]^d)^{\mathbb Z},\sigma).
\]
Here ``almost embedding'' means that, for every invariant probability measure
$\mu$ on $Y$, the map induces a measure-theoretic isomorphism between
$(Y,S,\mu)$ and its pushforward system
\cite[Theorem~10.1]{Meyerovitch2026}.
\end{enumerate}

For a continuous equivariant map $a:(Y,S)\to(W,R)$, write
\[
Y\times_aY=\{(u,v)\in Y\times Y:a(u)=a(v)\}
\]
for its fiber product.  Equivalently, $a$ is an almost embedding if and only
if every $(S\times S)$-invariant probability measure on $Y\times_aY$ is
supported on the diagonal $\Delta_Y$
\cite[Proposition~10.1]{Meyerovitch2026}.

We need the following inverse-limit estimate, which is analogous to the
mean-dimension estimate in \cref{prop:mdim-inverse} but uses invariant measures
explicitly.

\begin{lemma}\label{lem:Ddim-inverse}
Let
\[
(Y,S)=\varprojlim(Y_m,S_m)
\]
be an inverse limit of topological dynamical systems whose bonding maps are
surjective, and let $p_m:Y\to Y_m$ be the equivariant projections.  Then
\[
\Ddim(Y,S)\le \sup_m\Ddim(Y_m,S_m).
\]
\end{lemma}

\begin{proof}
Let $\mathcal U$ be a finite open cover of $Y$, and write
$\tau_{m,n}:Y_m\to Y_n$ for the bonding maps.  Since $Y$ has the subspace
topology inherited from $\prod_{m\ge1}Y_m$, open cylinder sets determined by
finitely many coordinates form a basis.  Thus, for each $y\in Y$, choose
$U(y)\in\mathcal U$, finitely many indices
$n_{y,1},\ldots,n_{y,r_y}$, and open sets
$V_{y,j}\subseteq Y_{n_{y,j}}$ such that
\[
B_y=\bigcap_{j=1}^{r_y}p_{n_{y,j}}^{-1}(V_{y,j})
\]
is an open cylinder satisfying $y\in B_y\subseteq U(y)$.  Hence
$\{B_y:y\in Y\}$ is an open cover of $Y$.  By compactness, there are points
$y_1,\ldots,y_\ell\in Y$ such that
\[
Y=\bigcup_{i=1}^{\ell}B_{y_i}.
\]
Set
\[
m=\max\{n_{y_i,j}:1\le i\le\ell,\ 1\le j\le r_{y_i}\}.
\]
Then all the selected cylinder sets depend only on coordinates at levels at
most $m$, and we may put
\[
A_i=\bigcap_{j=1}^{r_{y_i}}
\tau_{m,n_{y_i,j}}^{-1}(V_{y_i,j})\subseteq Y_m
\qquad(1\le i\le\ell).
\]
The identity $p_n=\tau_{m,n}\circ p_m$ gives
$B_{y_i}=p_m^{-1}(A_i)$.  Since $p_m$ is surjective and the $B_{y_i}$
cover $Y$, the sets $A_1,\ldots,A_\ell$ form a finite open cover
$\mathcal A$ of $Y_m$.  Moreover,
\begin{equation}\label{eq:single-stage-cover}
p_m^{-1}\mathcal A\succeq\mathcal U.
\end{equation}

It follows that
\[
\Ddim(\mathcal U;Y,S)
\le
\Ddim(p_m^{-1}\mathcal A;Y,S).
\]
Let $\mathcal B\succeq\mathcal A$ be a finite open cover of $Y_m$.  Then
$p_m^{-1}\mathcal B\succeq p_m^{-1}\mathcal A$, and for every $\mu\in\Prob(Y,S)$,
\begin{align*}
\int_Y\ord(p_m^{-1}\mathcal B,y)\,d\mu(y)
&=
\int_{Y_m}\ord(\mathcal B,z)\,d((p_m)_*\mu)(z)\\
&\le
\sup_{\nu\in\Prob(Y_m,S_m)}
\int_{Y_m}\ord(\mathcal B,z)\,d\nu(z).
\end{align*}
The pushforward $(p_m)_*\mu$ is $S_m$-invariant because $p_m$ is
equivariant.  For each $\mathcal B\succeq\mathcal A$, the pullback
$p_m^{-1}\mathcal B$ is an admissible refinement in the definition of
$\Ddim(p_m^{-1}\mathcal A;Y,S)$.  Therefore
\begin{align*}
\Ddim(p_m^{-1}\mathcal A;Y,S)
&\le
\inf_{\mathcal B\succeq\mathcal A}
\sup_{\mu\in\Prob(Y,S)}
\int_Y\ord(p_m^{-1}\mathcal B,y)\,d\mu(y)\\
&\le
\inf_{\mathcal B\succeq\mathcal A}
\sup_{\nu\in\Prob(Y_m,S_m)}
\int_{Y_m}\ord(\mathcal B,z)\,d\nu(z)\\
&=
\Ddim(\mathcal A;Y_m,S_m)
\le
\Ddim(Y_m,S_m).
\end{align*}
Combining this with the preceding inequality, we obtain
\[
\Ddim(\mathcal U;Y,S)
\le
\Ddim(Y_m,S_m)
\le
\sup_k\Ddim(Y_k,S_k).
\]
Although the index $m$ may depend on $\mathcal U$, the last upper bound does
not.  Therefore, taking the supremum over all finite open covers
$\mathcal U$ of $Y$ gives
\[
\Ddim(Y,S)
=
\sup_{\mathcal U\in\Cov(Y)}\Ddim(\mathcal U;Y,S)
\le
\sup_k\Ddim(Y_k,S_k),
\]
as required.
\end{proof}

\section{Dynamical dimension of full shifts and of the inverse-limit system}\label{sec:Ddim-full-shift}

The inverse-limit estimate reduces the desired upper bound to estimates on
the finite stages.  Since each stage lies in a full shift over $\Sone^N$, we
next extend \eqref{eq:Ddim-cube} from cubical alphabets to arbitrary
finite-dimensional compact metrizable alphabets.  Together with the fixed-point
lower bound, this also settles the infinite-dimensional case.  The key
observation for the finite-dimensional upper bound is that zero-dimensional
extensions do not increase dynamical dimension.

\begin{lemma}\label{lem:Ddim-zero-extension}
Let $p:(Y,S)\to(Z,R)$ be a factor map.  If every fiber of $p$ has covering
dimension zero,
then
\[
\Ddim(Y,S)\le \Ddim(Z,R).
\]
\end{lemma}

\begin{proof}
Fix a finite open cover $\mathcal U$ of $Y$, and let $z\in Z$.  Write
$F_z=p^{-1}(z)$.  Since $F_z$ is compact and zero-dimensional, every point
$x\in F_z$ has a clopen neighborhood $D_x$ in $F_z$ contained in
$U_x\cap F_z$ for some $U_x\in\mathcal U$.  Choose finitely many of the
$D_x$ that cover $F_z$.  After successively replacing each set by the part
not covered by the preceding ones and discarding empty sets, we obtain a
finite clopen partition subordinate to $\mathcal U$:
\[
F_z=C_{z,1}\sqcup\cdots\sqcup C_{z,r_z},
\qquad C_{z,i}\subset U_{z,i}\in\mathcal U.
\]
Each $C_{z,i}$ is compact.  Since $Y$ is normal, a finite shrinking argument
applied to the finitely many pairwise disjoint compact sets $C_{z,i}\subset
U_{z,i}$ gives pairwise disjoint open sets $O_{z,i}$ in $Y$ such that
\[
C_{z,i}\subset O_{z,i}\subset\overline{O_{z,i}}\subset U_{z,i}
\qquad(1\le i\le r_z).
\]
Put $O_z=\bigcup_{i=1}^{r_z}O_{z,i}$.  The map $p$ is closed because it is a
continuous map from a compact space to a Hausdorff space.  Hence
\[
V_z=Z\setminus p(Y\setminus O_z)
\]
is open.  Moreover, $F_z\subset O_z$, so
$z\notin p(Y\setminus O_z)$ and hence $z\in V_z$.  If $y\in p^{-1}(V_z)$
but $y\notin O_z$, then $p(y)\in p(Y\setminus O_z)$, contradicting
$p(y)\in V_z$.  Thus
\[
p^{-1}(V_z)\subset O_z.
\]
The sets $\{V_z:z\in Z\}$ form an open cover of the compact space $Z$.
Choose $z_1,\ldots,z_s$ such that
\[
\mathcal V=\{V_{z_1},\ldots,V_{z_s}\}
\]
still covers $Z$.

Let $\mathcal W\succeq\mathcal V$ be any finite open cover.  For every
$W\in\mathcal W$, choose $a(W)$ with $W\subset V_{z_{a(W)}}$, and define
\[
\widetilde{\mathcal W}
=
\left\{
p^{-1}(W)\cap O_{z_{a(W)},i}:
W\in\mathcal W,\ 1\le i\le r_{z_{a(W)}}
\right\},
\]
discarding empty sets.  We verify that this is an open cover of $Y$.  Given
$y\in Y$, choose $W\in\mathcal W$ with $p(y)\in W$.  Then
\[
y\in p^{-1}(W)
\subset p^{-1}(V_{z_{a(W)}})
\subset O_{z_{a(W)}}.
\]
Since the sets $O_{z_{a(W)},i}$ are pairwise disjoint and their union is
$O_{z_{a(W)}}$, there is a unique $i$ such that
$y\in p^{-1}(W)\cap O_{z_{a(W)},i}$.  Hence
$\widetilde{\mathcal W}$ covers $Y$.  Every member of
$\widetilde{\mathcal W}$ is contained in some
$O_{z_{a(W)},i}\subset U_{z_{a(W)},i}$, so
\[
\widetilde{\mathcal W}\succeq\mathcal U.
\]
For a fixed $y\in Y$ and each $W\in\mathcal W$ containing $p(y)$, exactly one
set of $\widetilde{\mathcal W}$ indexed by $W$ contains $y$.  Different
indices may define the same set after empty sets and repetitions are removed,
so the number of members of $\widetilde{\mathcal W}$ containing $y$ is no
larger than the number of members of $\mathcal W$ containing $p(y)$.
Therefore
\begin{equation}\label{eq:zero-extension-order}
\ord(\widetilde{\mathcal W},y)
\le \ord(\mathcal W,p(y))
\qquad(y\in Y).
\end{equation}

If $\mu\in\Prob(Y,S)$, equivariance of $p$ implies that
$p_*\mu\in\Prob(Z,R)$.  Integrating \eqref{eq:zero-extension-order} gives
\[
\int_Y\ord(\widetilde{\mathcal W},y)\,d\mu(y)
\le
\int_Z\ord(\mathcal W,z)\,d(p_*\mu)(z).
\]
Taking the supremum over $\mu\in\Prob(Y,S)$, and then enlarging the set of
measures on the right from the pushforwards $p_*\mu$ to all invariant
measures on $Z$, yields
\[
\sup_{\mu\in\Prob(Y,S)}
\int_Y\ord(\widetilde{\mathcal W},y)\,d\mu(y)
\le
\sup_{\nu\in\Prob(Z,R)}
\int_Z\ord(\mathcal W,z)\,d\nu(z).
\]
Since $\widetilde{\mathcal W}\succeq\mathcal U$, it is an admissible
refinement in the definition of $\Ddim(\mathcal U;Y,S)$.  Therefore, for
every $\mathcal W\succeq\mathcal V$,
\[
\Ddim(\mathcal U;Y,S)
\le
\sup_{\nu\in\Prob(Z,R)}
\int_Z\ord(\mathcal W,z)\,d\nu(z).
\]
Taking the infimum over all finite open covers
$\mathcal W\succeq\mathcal V$ gives
\[
\Ddim(\mathcal U;Y,S)
\le
\Ddim(\mathcal V;Z,R)
\le
\Ddim(Z,R).
\]
This bound holds for every finite open cover $\mathcal U$ of $Y$.  Taking
the supremum over $\mathcal U\in\Cov(Y)$ yields
\[
\Ddim(Y,S)\le\Ddim(Z,R).
\]
\end{proof}

The preceding lemma allows a finite-dimensional alphabet to be replaced by a
cubical image without increasing the dynamical dimension of its full shift.
Together with the fixed-point lower bound, this computes the invariant for
every compact metrizable alphabet, including infinite-dimensional ones.

\begin{theorem}\label{thm:Ddim-full-shift}
Let $K$ be a nonempty compact metrizable space.  Then
\[
\Ddim(K^{\mathbb Z},\sigma)=\dim K,
\]
where both sides are allowed to be infinite.
\end{theorem}

\begin{proof}
The fixed-point set of the shift consists of the constant configurations and
is therefore homeomorphic to $K$.  The fixed-point lower bound
\eqref{eq:Ddim-fixed-point} gives
\[
\dim K
=
\dim\operatorname{Fix}(\sigma)
\le
\Ddim(K^{\mathbb Z},\sigma).
\]
If $\dim K=\infty$, this already proves the theorem.

Suppose now that $d:=\dim K<\infty$.  It remains to prove the upper bound.  If
$d\ge1$, apply the Pasynkov--Toru\'nczyk
mapping theorem to the constant map $f:K\to\{\mathrm{pt}\}$.  Here
$\dim f$ denotes the supremum of the covering dimensions of the fibers, so
$\dim f=\dim K=d$.  Thus
\cite[Theorem~1.7, $(1)\Rightarrow(3')$]{LevinLewis} gives a continuous map
$q:K\to[0,1]^d$ such that
$(f,q):K\to\{\mathrm{pt}\}\times[0,1]^d$ is zero-dimensional.  Its fibers
are precisely the fibers of $q$, so all fibers of $q$ are zero-dimensional.
If $d=0$, let $q$ be the map from $K$ to the one-point space $[0,1]^0$; its
only fiber is $K$, which is zero-dimensional.  In either case, put $L=q(K)$
and regard $q$ as a surjection onto $L$.  The coordinatewise map
\[
Q=q^{\mathbb Z}:K^{\mathbb Z}\longrightarrow L^{\mathbb Z}
\]
is a factor map.  For every
$y=(y_n)_{n\in\mathbb Z}\in L^{\mathbb Z}$, its fiber is
\[
Q^{-1}(y)
=\prod_{n\in\mathbb Z}q^{-1}(y_n).
\]
Each factor $q^{-1}(y_n)$ is a zero-dimensional compact metrizable space.
Basic open sets in the product depend on only finitely many coordinates, and
in each of those coordinates they can be refined by clopen sets.  The
resulting finite-coordinate cylinders are clopen and form a basis, so
$Q^{-1}(y)$ is zero-dimensional.  Therefore
\cref{lem:Ddim-zero-extension} gives
\[
\Ddim(K^{\mathbb Z},\sigma)
\le
\Ddim(L^{\mathbb Z},\sigma).
\]
Since $L\subseteq[0,1]^d$, the coordinatewise inclusion is an equivariant
topological embedding
\[
L^{\mathbb Z}\hookrightarrow([0,1]^d)^{\mathbb Z}.
\]
Monotonicity under equivariant embeddings and \eqref{eq:Ddim-cube} now give
\[
\Ddim(L^{\mathbb Z},\sigma)
\le
\Ddim(([0,1]^d)^{\mathbb Z},\sigma)
=d.
\]
Combining the upper and lower bounds completes the proof.
\end{proof}

\begin{remark}
For finite-dimensional alphabets, \cref{thm:Ddim-full-shift} contrasts with
Tsukamoto's formula for mean dimension \cite[Theorem~1.1]{Tsukamoto2019}:
\[
\mdim(K^{\mathbb Z},\sigma)
=
\begin{cases}
\dim K, & \text{if $K$ is of basic type},\\
\dim K-1, & \text{if $K$ is of exceptional type}.
\end{cases}
\]
Thus dynamical dimension and mean dimension agree for full shifts over
alphabets of basic type, whereas they differ by one for alphabets of
exceptional type.  Tsukamoto's theorem assumes that $K$ is
finite-dimensional.  For a general infinite-dimensional compact metrizable
alphabet $K$, the value of $\mdim(K^{\mathbb Z},\sigma)$ remains an open
problem.  In contrast, \cref{thm:Ddim-full-shift} shows that
$\Ddim(K^{\mathbb Z},\sigma)=\infty$ whenever $\dim K=\infty$.
\end{remark}

\begin{remark}
Meyerovitch's definition and cubical full-shift formula apply to every
countable acting group.  The same proof therefore shows that, for any
countable group $G$,
\[
\Ddim(K^G,\mathrm{shift})=\dim K.
\]
\end{remark}

We now combine the full-shift formula with the inverse-limit estimate and the
mean-dimension lower bound obtained in \cref{prop:inverse-mdim}.

\begin{proposition}\label{prop:Ddim-bound}
The dynamical system satisfies
\[
\Ddim(\calX,T)=N.
\]
Consequently, there is an equivariant almost embedding
\begin{equation}\label{eq:almost-A}
(\calX,T)
\longrightarrow
\left(([0,1]^{2N+1})^{\mathbb Z},\sigma\right).
\end{equation}
\end{proposition}

\begin{proof}
Since $\dim(\Sone^N)=N$, \cref{thm:Ddim-full-shift} and monotonicity under
equivariant embeddings give
\[
\Ddim(X_m,T_m)\le
\Ddim((\Sone^N)^{\mathbb Z},\sigma)=N.
\]
By construction,
\[
(\calX,T)=\varprojlim(X_m,T_m).
\]
Thus, by \cref{lem:Ddim-inverse},
\[
\Ddim(\calX,T)
\le
\sup_m\Ddim(X_m,T_m)
\le N.
\]
On the other hand, \cref{prop:inverse-mdim} and
\cite[Theorem~9.1]{Meyerovitch2026} give
\[
N=\mdim(\calX,T)\le\Ddim(\calX,T).
\]
Hence
$\Ddim(\calX,T)=N$.

Since
\[
\Ddim(\calX,T)=N<\frac{2N+1}{2},
\]
\cite[Theorem~10.1]{Meyerovitch2026} gives
\eqref{eq:almost-A}.
\end{proof}

\section{The first-coordinate projection is a distal extension}\label{sec:relative-distal}

We next prove the special structural property that upgrades the almost embedding to an actual embedding.

\begin{definition}
A factor map $p:(Y,S)\to(Z,R)$ is called a
\emph{distal extension} if, whenever $y_1\ne y_2$ and
$p(y_1)=p(y_2)$, the orbit closure of $(y_1,y_2)$ under $S\times S$ is
disjoint from the diagonal $\Delta_Y$.
\end{definition}

Thus a distal extension is the relative analogue of a distal system.  Indeed,
$(Y,S)$ is distal if and only if the unique factor map from $(Y,S)$ to the
one-point system is a distal extension.  More generally, every factor map
from a distal system is a distal extension; conversely, if
$p:(Y,S)\to(Z,R)$ is a distal extension and $(Z,R)$ is distal, then $(Y,S)$
is distal.

We now return to the system $(\calX,T)$ constructed in
\cref{sec:construction}.  To verify that its first-coordinate projection
$\pi_1:\calX\to X_1$ is a distal extension, we first derive an explicit
formula for the composite bonding map from $X_m$ to $X_1$.

Recall that, although the bonding maps
$\theta_{j,j-1}:X_j\to X_{j-1}$ were originally defined between the finite
stages, the formula in \eqref{eq:bonding} defines group homomorphisms on the
ambient full shift $(\Sone^N)^{\mathbb Z}$.  We use the same notation
$\theta_{j,k}$ for the corresponding composites on the full shift.

\begin{lemma}\label{lem:theta-telescope}
For every $m\ge1$ and every $z\in(\Sone^N)^{\mathbb Z}$,
\begin{equation}\label{eq:theta-sum}
(\theta_{m,1}z)_k
=
\sum_{i=0}^{q(m)-1}z_{k+i}
\qquad(k\in\mathbb Z).
\end{equation}
\end{lemma}

\begin{proof}
For $m=1$, the assertion follows from $q(1)=1$ and
$\theta_{1,1}=\id$.  We proceed by induction on $m$.  Suppose that the
formula holds for $m-1$, where $m\ge2$.  Since
\[
\theta_{m,1}
=
\theta_{m-1,1}\circ\theta_{m,m-1},
\]
the induction hypothesis and \eqref{eq:bonding} give, for every
$z\in(\Sone^N)^{\mathbb Z}$ and $k\in\mathbb Z$,
\begin{align*}
(\theta_{m,1}z)_k
&=
\sum_{a=0}^{q(m-1)-1}(\theta_{m,m-1}z)_{k+a}\\
&=
\sum_{a=0}^{q(m-1)-1}\sum_{b=0}^{m-1}
z_{k+a+bq(m-1)}.
\end{align*}
Every integer $i$ with $0\le i<q(m)=m q(m-1)$ has a unique
representation
\[
i=a+bq(m-1),
\qquad
0\le a<q(m-1),\quad 0\le b<m.
\]
Consequently, the last double sum is
\[
\sum_{i=0}^{q(m)-1}z_{k+i},
\]
which proves \eqref{eq:theta-sum} and completes the induction.
\end{proof}

The telescoping formula turns equality in the first coordinate into periodicity of
the difference at every higher stage.  This gives the separation from the
diagonal required for distality.

\begin{proposition}\label{prop:relative-distal}
The projection
\[
\pi_1:(\calX,T)\longrightarrow(X_1,T_1)
\]
is a distal extension.
\end{proposition}

\begin{proof}
Write points of the inverse limit as
\[
\mathbf x=(x^{(1)},x^{(2)},\ldots),
\qquad
\mathbf y=(y^{(1)},y^{(2)},\ldots).
\]
Suppose that $\mathbf x\ne\mathbf y$ but
\[
\pi_1(\mathbf x)=\pi_1(\mathbf y).
\]
Choose $m\ge2$ such that $x^{(m)}\ne y^{(m)}$, and define
\[
z=x^{(m)}-y^{(m)}\in(\Sone^N)^{\mathbb Z}.
\]
By the inverse-limit compatibility relations and the equality
$\pi_1(\mathbf x)=\pi_1(\mathbf y)$,
\[
\theta_{m,1}(x^{(m)})
=x^{(1)}
=y^{(1)}
=\theta_{m,1}(y^{(m)}).
\]
Using the group-homomorphism extension of $\theta_{m,1}$ recalled above, we
obtain
\[
\theta_{m,1}(z)
=
\theta_{m,1}(x^{(m)})-\theta_{m,1}(y^{(m)})
=0.
\]
By \cref{lem:theta-telescope}, for every $k\in\mathbb Z$,
\begin{equation}\label{eq:zero-block-sum}
\sum_{i=0}^{q(m)-1}z_{k+i}=0.
\end{equation}
Applying \eqref{eq:zero-block-sum} at $k$ and $k+1$, we obtain
\[
0
=
\sum_{i=0}^{q(m)-1}z_{k+1+i}
-
\sum_{i=0}^{q(m)-1}z_{k+i}
=
z_{k+q(m)}-z_k.
\]
Therefore
\begin{equation}\label{eq:z-periodic}
z_{k+q(m)}=z_k
\qquad(k\in\mathbb Z).
\end{equation}
Since $x^{(m)}\ne y^{(m)}$, the sequence $z$ is nonzero.  The preceding
identity \eqref{eq:z-periodic} shows that $q(m)$ is a period of $z$ (and
hence that the least period of $z$ divides $q(m)$).  Therefore the shift
orbit of $z$ is the finite set
\[
F=\{\sigma^r z:0\le r<q(m)\}.
\]
Because $\sigma$ is invertible and $z\ne0$, no element of $F$ is the zero
sequence.

Let $K$ be the orbit closure of $(\mathbf x,\mathbf y)$ under $T\times T$.
Since $\calX\times\calX$ is compact, so is $K$.  Consider the continuous map
\[
\Phi_m:\calX\times\calX\longrightarrow(\Sone^N)^{\mathbb Z},
\qquad
\Phi_m(\mathbf u,\mathbf v)=u^{(m)}-v^{(m)}.
\]
For every $n\in\mathbb Z$, equivariance of $\pi_m$ gives
\[
\pi_m(T^n\mathbf x)=\sigma^n x^{(m)}
\quad\text{and}\quad
\pi_m(T^n\mathbf y)=\sigma^n y^{(m)}.
\]
Since the shift $\sigma$ is a group automorphism of
$(\Sone^N)^{\mathbb Z}$, it follows that
\[
\begin{aligned}
\Phi_m(T^n\mathbf x,T^n\mathbf y)
&=
\sigma^n x^{(m)}-\sigma^n y^{(m)}\\
&=
\sigma^n(x^{(m)}-y^{(m)})\\
&=
\sigma^n z
\in F.
\end{aligned}
\]
Since $F$ is closed, continuity of $\Phi_m$ implies
$\Phi_m(K)\subseteq F$.  On the other hand,
$\Phi_m(\mathbf w,\mathbf w)=0$ for every $\mathbf w\in\calX$, whereas
$0\notin F$.  Thus $K\cap\Delta_{\calX}=\varnothing$, proving that $\pi_1$
is a distal extension.
\end{proof}

\section{Upgrading the almost embedding}\label{sec:upgrade}

We now combine the almost embedding from \cref{prop:Ddim-bound} with the
distal factor found in \cref{prop:relative-distal}.  The following general
lemma is a relative version of the fact that an almost embedding of a distal
system is an embedding \cite[Proposition~10.2]{Meyerovitch2026}.  It isolates
the mechanism that makes the combined map injective.

\begin{lemma}\label{lem:upgrade}
Let $a:(Y,S)\to(W,R)$ be an equivariant almost embedding, and let
$p:(Y,S)\to(Z,Q)$ be a distal extension.  Then
\[
(a,p):Y\longrightarrow W\times Z
\]
is injective, and hence is a topological embedding when $Y$ is compact and $W\times Z$ is Hausdorff.
\end{lemma}

\begin{proof}
Suppose
\[
a(y_1)=a(y_2),
\qquad
p(y_1)=p(y_2).
\]
Assume for contradiction that $y_1\ne y_2$.  Let $K$ be the orbit closure of
$(y_1,y_2)$ under $S\times S$.  Equivariance of $a$ and the equality
$a(y_1)=a(y_2)$ imply
\[
K\subseteq Y\times_aY.
\]
Distality of the extension $p$ and the equality $p(y_1)=p(y_2)$ imply
\[
K\cap\Delta_Y=\varnothing.
\]
$K$ is a nonempty compact invariant set for $S\times S$.  Choose
$\xi\in K$ and, for $n\ge1$, define the empirical probability measure
\[
\mu_n=\frac1n\sum_{j=0}^{n-1}\delta_{(S\times S)^j\xi}
\]
on $K$.  The space of Borel probability measures on the compact metrizable
space $K$ is weak-$*$ compact, so some subsequence $\mu_{n_\ell}$ converges
weak-$*$ to a Borel probability measure $\mu$ on $K$.  As a weak-$*$
limit of empirical measures, $\mu$ is $(S\times S)$-invariant.
Regarding $\mu$ as a probability measure on $Y\times Y$, we have
$\mu(K)=1$.  Because
$K\cap\Delta_Y=\varnothing$, this measure, viewed as a measure on the fiber
product $Y\times_aY$, is not supported on the diagonal.  This contradicts the
fiber-product characterization of an almost embedding
\cite[Proposition~10.1]{Meyerovitch2026}.  Therefore $y_1=y_2$.
\end{proof}

To apply \cref{lem:upgrade}, it remains to encode the first-coordinate factor
in a finite-dimensional cubical shift.  For this we use the following
elementary embedding of its alphabet.

\begin{lemma}\label{lem:torus-embedding}
For every $N\ge1$, the torus $\Sone^N$ admits a smooth embedding into
$[0,1]^{N+1}$.
\end{lemma}

\begin{proof}
It is enough to construct an embedding in $\mathbb R^{N+1}$ and then compose
it with an affine rescaling into the interior of $[0,1]^{N+1}$.  We argue
inductively.  The case $N=1$ is the usual circle in $\mathbb R^2$.  Suppose
that $\Sone^{N-1}$ is embedded as an orientable smooth hypersurface in
$\mathbb R^N$.  Regard it as a codimension-two submanifold of
$\mathbb R^{N+1}=\mathbb R^N\times\mathbb R$.  Its normal
bundle is the direct sum of the trivial normal line bundle in $\mathbb R^N$
and the last coordinate line, and hence is a trivial rank-two bundle.  The
boundary of a sufficiently small tubular neighborhood is therefore embedded
in $\mathbb R^{N+1}$ and is homeomorphic to
\[
\Sone^{N-1}\times\mathbb S^1\cong\Sone^N.
\]
This boundary is again an orientable hypersurface, completing the induction.
\end{proof}

All ingredients for the final embedding are now in place.

\begin{proof}[Completion of the proof of \cref{thm:main}]
Properties~(i) and~(ii) were established in
\cref{sec:construction}, while property~(iii) follows from
\cref{prop:inverse-mdim,prop:Ddim-bound}.  It remains to prove the embedding
assertion in property~(iv).

By \cref{prop:Ddim-bound}, we may choose an equivariant almost embedding
\[
A:(\calX,T)\longrightarrow
\left(([0,1]^{2N+1})^{\mathbb Z},\sigma\right).
\]
Let
\[
E_1:X_1\longrightarrow([0,1]^{N+1})^{\mathbb Z}
\]
be the coordinatewise embedding induced by \cref{lem:torus-embedding}.  More
explicitly, if $e:\Sone^N\hookrightarrow[0,1]^{N+1}$ is the chosen embedding,
then $(E_1x)_k=e(x_k)$.  By
\cref{prop:relative-distal}, $\pi_1$ is a distal extension.  Therefore
\cref{lem:upgrade} shows that $(A,\pi_1)$ is injective.  Since $E_1$ is
injective, it follows that
\[
\Phi=(A,E_1\circ\pi_1):
\calX
\longrightarrow
([0,1]^{2N+1})^{\mathbb Z}
\times
([0,1]^{N+1})^{\mathbb Z}
\]
is injective.  It is continuous and equivariant.  Concatenating the two
alphabet coordinates gives a canonical conjugacy
\[
([0,1]^{2N+1})^{\mathbb Z}
\times
([0,1]^{N+1})^{\mathbb Z}
\cong
([0,1]^{3N+2})^{\mathbb Z}.
\]
Since $\calX$ is compact, the resulting continuous injection into the cubical
shift is a topological embedding.
\end{proof}

\section{Open questions}\label{sec:open-questions}

The preceding argument leaves several natural questions.

\begin{enumerate}[label=(\arabic*)]
\item Can the alphabet dimension in \cref{thm:main} be reduced from $3N+2$ to
$2N+1$?  The almost embedding already takes values in the
$(2N+1)$-cubical shift; the additional $N+1$ coordinates are used only to
encode the distal extension $\pi_1$.  An affirmative
answer would therefore require either a direct upgrade of this almost
embedding or a way of incorporating the factor information without
increasing the alphabet dimension.

\item Does every $\mathbb Z$-dynamical system $(Y,S)$ with
$\Ddim(Y,S)<\infty$ admit an equivariant topological embedding into
$(([0,1]^D)^{\mathbb Z},\sigma)$ for some finite $D$?  Finite dynamical
dimension gives an almost embedding by \cite{Meyerovitch2026}, but in the
present example the passage to a topological embedding also uses the special
distal extension $\pi_1$.  It is unclear what could replace this
additional structure in general.

\item For which aperiodic systems without the marker property does one have
\[
\mdim(Y,S)=\Ddim(Y,S)?
\]
The system studied here shows that the equality can hold even in the absence
of the marker property.  At the opposite extreme, Meyerovitch observed that
the free isometric system of Dranishnikov and Levin
\cite{DranishnikovLevin2025,Meyerovitch2026} satisfies
\[
\mdim(Y,S)=0
\qquad\text{and}\qquad
\Ddim(Y,S)=\infty.
\]
It remains natural to seek general hypotheses ensuring equality without the
marker property and, in particular, to ask whether there is a free non-marker
$\mathbb Z$-system for which both invariants are finite but
\[
\mdim(Y,S)<\Ddim(Y,S)<\infty.
\]
\end{enumerate}

\bigskip
\noindent Shanghai Center for Mathematical Sciences, Fudan University,\\
200438 Shanghai, China\\
Email: \href{mailto:ruxishi@fudan.edu.cn}{ruxishi@fudan.edu.cn}

\end{document}